\documentclass[11pt]{amsart}

\usepackage{amsthm}
\usepackage{amsmath}
\usepackage{amssymb}

\usepackage{enumerate}
\usepackage{tikz}

\usepackage{float}
\usepackage{subfigure}

\usepackage[colorlinks = true, linkcolor = black, citecolor = black, final]{hyperref}

\newtheorem{thm}{Theorem}[section]
\newtheorem{prop}[thm]{Proposition}
\newtheorem{cor}[thm]{Corollary}
\newtheorem{lemma}[thm]{Lemma}

	\renewcommand{\L}{\Lambda}
	
	\renewcommand{\l}{\lambda}
	
	\newcommand{\Okm}{\Omega_{k,m}}
	
	\newcommand{\NN}{\mathbb{N}}
	\newcommand{\ZZ}{\mathbb{Z}}
	\newcommand{\TT}{\mathbb{T}}	
	
	\DeclareMathOperator{\Obj}{Obj}
	\DeclareMathOperator{\Mor}{Mor}
	\DeclareMathOperator{\MCE}{MCE}
	
	\def\IoIIdimdots(#1/#2/#3,#4){\node at (#1,#4) {\color{gray}$.$};\node at (#2,#4) {\color{gray}$.$};\node at (#3,#4) {\color{gray}$.$};}
	\def\IIoIIdimdots(#1,#2/#3/#4){\node at (#1,#2) {\color{gray}$.$};\node at (#1,#3) {\color{gray}$.$};\node at (#1,#4) {\color{gray}$.$};}

	\def\IoIIIdimdots(#1/#2/#3,#4,#5){\node at (#1,#4,#5) {\color{gray}$.$};\node at (#2,#4,#5) {\color{gray}$.$};\node at (#3,#4,#5) {\color{gray}$.$};}
	\def\IIoIIIdimdots(#1,#2/#3/#4,#5){\node at (#1,#2,#5) {\color{gray}$.$};\node at (#1,#3,#5) {\color{gray}$.$};\node at (#1,#4,#5) {\color{gray}$.$};}
	\def\IIIoIIIdimdots(#1,#2,#3/#4/#5){\node at (#1,#2,#3) {\color{gray}$.$};\node at (#1,#2,#4) {\color{gray}$.$};\node at (#1,#2,#5) {\color{gray}$.$};}

\begin{document}

\title{Aperiodicity Conditions in Topological $k$-Graphs}

\author{Sarah Wright}
\address{College of the Holy Cross, Department of Mathematics and Computer Science, One College Street, Worcester, MA, 01610}
\email{swright@holycross.edu}

\begin{abstract}

We give two new conditions on topological $k$-graphs that are equivalent to the Yeend's aperiodicity Condition (A).  Each of the new conditions concerns finite paths rather than infinite.  We use a specific example, resulting from a new construction of a twisted topological $k$-graph, to demonstrate the improvements achieved by the new conditions.  Reducing this proof of equivalence to the discrete case also gives a new direct proof of the corresponding conditions in discrete $k$-graphs, where previous proofs depended on simplicity of the corresponding C$^*$-algebra.

\end{abstract}



\maketitle

\section{Introduction}

The theory of graph C$^*$-algebras began with the work of Cuntz and Kreiger \cite{Cuntz-Krieger} and the later work of Enomoto and Watatani \cite{EnomotoWatatani}.  Since then, there have been many contributions by a variety of researchers resulting in an extensive collection of literature.  The main idea is to associate to each directed graph a C$^*$-algebra and use the combinatorics of the graph to answer questions about the structure of the C$^*$-algebra.  The study of graph algebras had been particularly fruitful in its provision of a rich class of easily accessible examples of C$^*$-algebras with various properties.

To this end, the idea of a graph has been generalized in a few ways including the $k$-graphs of Kumjian and Pask \cite{KP} and the topological graphs of Katsura \cite{Katsura1}.  Most recently, the work of Yeend provides a generalization of both higher-rank graphs and topological graphs, with the unifying theory of topological $k$-graphs \cite{Yeend1}.

An important outcome of the study of each type of graph is the relationship between the periodic paths (or lack there of) in the graph and the simplicity of the associated C$^*$-algebra.  The first such aperiodicity condition for directed graphs appears in \cite{KPRR}, there referred to as Condition (L), and states that every cycle of the graph must have an entry.  This is one of the necessary conditions for simplicity and that relationship demonstrates the beauty of the subject; it is easy to look at picture of a directed graph and check  if each cycle has an entry.  

The key to cycles having or not having entries is the idea of being able to ``back out of a cycle" and build an aperiodic infinite path.  As the infinite paths play an important role in the groupoid construction of the C$^*$-algebra, in the case of higher-rank graphs the original aperiodicity condition, called Condition (A), was stated in terms of infinite paths.  Lewin, Robertson, and Sims have developed conditions on higher-rank graphs, which are equivalent to aperiodicity \cite{LS, RandS}, that involve only finite paths.  The work of this paper gives extensions of these aperiodicity conditions to the case of topological $k$-graphs and proves their equivalence.

In the first section we'll discuss the background on topological $k$-graphs that is needed.  In the middle section we state the two new aperiodicity conditions and prove the main result of equivalence with Yeend's Condition (A).  We also address a new proof of the equivalence of the corresponding discrete conditions.  The final section gives a new method of constructing a topological $k$-graph from a discrete $k$-graph.  The construction of these twisted topological $k$-graphs shares the flavor of Yeend's skew product graphs \cite[Definition 8.1]{Yeend1} as well as the topological dynamical systems defined by Farthing, Patani, and Willis, \cite{FPW}.  An advantage over the skew product graphs is the fact that we need not begin with a topological $k$-graph, but can use a discrete graph with desired properties.  Also, we need only a suitable topological space and a continuous functor rather than the $k$ commuting local homeomorphisms $\{T_i\}$ of topological dynamical systems.  We give a specific example of this construction and use one of the new conditions to show the topological $k$-graph is aperiodic, demonstrating the improvements gained by considering finite paths.

{\bf Acknowledgements:}  I would like to give special thanks to Aidan Sims for many helpful conversations, ideas, and motivation.  I would also like to thank an anonymous referee for pointing out the errors in the original statement of the main theorem as well as many insightful comments.

\section{Background}

The basics (and considerably more) on topological $k$-graphs and their C$^*$-algebras can be found in Yeend's original papers \cite{Yeend2, Yeend1, YeendThesis}.  We review some of the basics here for convenience.

We will consider $\NN$ to contain 0 and regard $\NN^k$ as the category with a single object and composition given by addition.  We use $\{e_i\}_{i=1}^k$ to represent the standard basis of $\NN^k$ and for $m \in \NN^k$ denote the $i^{\text{th}}$ component by $m_i$.  For $m, n \in \NN^k$, we say $m \geq n$ if $m_i \geq n_i$ for every $i \in \{1, 2, \dots, k\}$, write $m \vee n$ for the coordinate maximum, and $m \wedge n$ for the coordinate minimum.  

For a natural number $k$, a \emph{topological $k$-graph} is a pair $(\L, d)$  consisting of a small category $\L = \left(\Obj(\L), \Mor(\L), r, s\right)$ and a functor $d:\L \rightarrow \NN^k$ which satisfy the following:
\begin{enumerate}
	\item  $\Obj(\L)$ and $\Mor(\L)$ are second-countable, locally compact, Hausdorff spaces;
	\item  $r, s: \Mor(\L) \rightarrow \Obj(\L)$ are continuous and $s$ is a local homeomorphism;
	\item  Composition $\circ: \L \times_c \L \rightarrow \L$ is continuous and open, where $\L \times_c \L$ has the relative topology inherited from the product topology on $\L \times \L$;
	\item  $d: \L \rightarrow \NN^k$ is continuous, where $\NN^k$ has the discrete topology; and
	\item  \label{UFP}For all $\l \in \L$ and all $m, n \in \NN^k$ such that $d(\l) = m + n$, there exist unique $\mu, \nu \in \L$ such that $d(\mu) = m$, $d(\nu) = n$, and $\l = \mu\nu$.
\end{enumerate}

Condition~\eqref{UFP} is called the \emph{unique factorization property} and is exactly the same as that for discrete $k$-graphs.  The functor $d$ is called the \emph{degree map}, or the \emph{shape map}.  For $n \in \NN^k$, we denote by $\L^n$ the open set $d^{-1}(n)$ in $\Mor(\L)$, and refer to its elements as \emph{paths of shape $n$}.  It is often necessary to deal with the range and source maps as they relate to paths of a specific shape, so the notation $r_n$ refers to the restriction of the range map $r$ to the set $\L^n$ and similarly $s_n := s|_{\L^n}$.  A consequence of the unique factorization property is that $\L^0$ is the set of identity morphisms of $\L$, and these are referred to as the \emph{vertices} of $\L$.  For any sets $X, Y \subset \L$ we use $XY$ for the set $\{\mu\nu \,|\, \mu \in X, \nu \in Y , s(\mu) = r(\nu)\}$.  This notation is of particular use when the first set is a set of vertices and the second set is all the paths.  So, for $V \subset \L^0$, $V\L = \{ \l \in \L \,|\,r(\l) \in V\}$ and $\L V = \{ \l \in \L \,|\, s(\l) \in V\}$.  A vertex $v$ is a \emph{source} in $\L$ if for some $n \in \NN^k$ $v\L^n$ is empty.

The factorization property implies that for $\l \in \L$ and $m < n < d(\l) \in \NN^k$ there are unique paths denoted $\l(0,m) \in \L^m$, $\l(m,n)\in \L^{n-m}$, and $\l\left(n,d(\l)\right) \in \L^{d(\l) - n}$ such that $\l(0,m)\l(m,n)\l(n,d(\l)) = \l$.  We think of $\l(p,q)$ as the portion of the path $\l$ that runs from $q$ to $p$.

An important class of examples of $k$-graphs are the \emph{grid graphs}, $\Okm$.  For a fixed $k \geq 1$ and $m \in \left(\NN \cup \{\infty\}\right)^k$ the discrete $k$-graph $\Okm$ has morphisms $\{(p,q) \in \NN^k \times \NN^k\,|\,p \leq q \leq m\}$, range and source maps given by $r(p,q) = (p,p)$ and $s(p,q) = (q,q)$, composition defined as $(p,q)(q,r) = (p,r)$ and degree map given by $d(p,q) = q-p$.

We visualize a discrete $k$-graph by its \emph{1-skeleton}, a directed graph whose vertices are $\L^0$ and whose edges are paths from the sets $\L^{e_i}$ and are colored $k$ different colors depending on the shape.  Frequently dashed or dotted arrows are also (or instead) used to depict edges of different shape.  For consistency, in this paper green edges will be solid and of shape $e_1$, blue edges dotted and of shape $e_2$, and red edges dashed and of shape $e_3$.
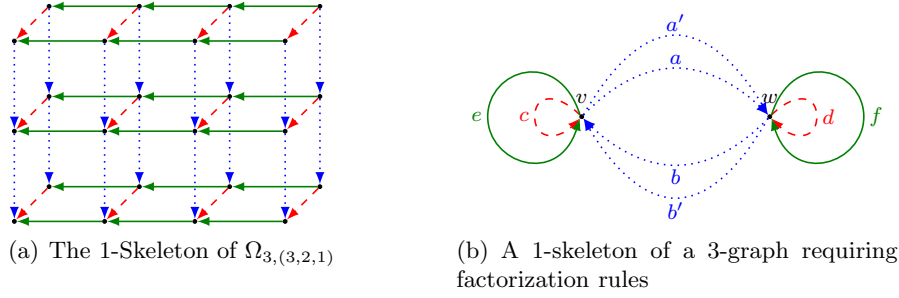
\begin{figure}[h]
\centering
\subfigure[The 1-Skeleton of $\Omega_{3, (3,2,1)}$]{
\begin{tikzpicture}[scale=1.2]

 \node[inner sep=.5pt, circle, fill=black] (vert000) at (0,0,0) [draw] {};
 \node[inner sep=.5pt, circle, fill=black] (vert100) at (1,0,0) [draw] {};
 \node[inner sep=.5pt, circle, fill=black] (vert200) at (2,0,0) [draw] {};
 \node[inner sep=.5pt, circle, fill=black] (vert300) at (3,0,0) [draw] {};
 
 \node[inner sep=.5pt, circle, fill=black] (vert001) at (0,0,1) [draw] {};
 \node[inner sep=.5pt, circle, fill=black] (vert101) at (1,0,1) [draw] {};
 \node[inner sep=.5pt, circle, fill=black] (vert201) at (2,0,1) [draw] {};
 \node[inner sep=.5pt, circle, fill=black] (vert301) at (3,0,1) [draw] {};
 
 \node[inner sep=.5pt, circle, fill=black] (vert010) at (0,1,0) [draw] {};
 \node[inner sep=.5pt, circle, fill=black] (vert110) at (1,1,0) [draw] {};
 \node[inner sep=.5pt, circle, fill=black] (vert210) at (2,1,0) [draw] {};
 \node[inner sep=.5pt, circle, fill=black] (vert310) at (3,1,0) [draw] {};
 
 \node[inner sep=.5pt, circle, fill=black] (vert011) at (0,1,1) [draw] {};
 \node[inner sep=.5pt, circle, fill=black] (vert111) at (1,1,1) [draw] {};
 \node[inner sep=.5pt, circle, fill=black] (vert211) at (2,1,1) [draw] {};
 \node[inner sep=.5pt, circle, fill=black] (vert311) at (3,1,1) [draw] {}; 
 
 \node[inner sep=.5pt, circle, fill=black] (vert020) at (0,2,0) [draw] {};
 \node[inner sep=.5pt, circle, fill=black] (vert120) at (1,2,0) [draw] {};
 \node[inner sep=.5pt, circle, fill=black] (vert220) at (2,2,0) [draw] {};
 \node[inner sep=.5pt, circle, fill=black] (vert320) at (3,2,0) [draw] {};
 
 \node[inner sep=.5pt, circle, fill=black] (vert021) at (0,2,1) [draw] {};
 \node[inner sep=.5pt, circle, fill=black] (vert121) at (1,2,1) [draw] {};
 \node[inner sep=.5pt, circle, fill=black] (vert221) at (2,2,1) [draw] {};
 \node[inner sep=.5pt, circle, fill=black] (vert321) at (3,2,1) [draw] {};

 \draw[style=semithick, green!50!black, -latex] (vert120.west)--(vert020.east);
 \draw[style=semithick, green!50!black, -latex] (vert110.west)--(vert010.east);
 \draw[style=semithick, green!50!black, -latex] (vert100.west)--(vert000.east);
 \draw[style=semithick, green!50!black, -latex] (vert220.west)--(vert120.east);
 \draw[style=semithick, green!50!black, -latex] (vert210.west)--(vert110.east);
 \draw[style=semithick, green!50!black, -latex] (vert200.west)--(vert100.east);
 \draw[style=semithick, green!50!black, -latex] (vert320.west)--(vert220.east);
 \draw[style=semithick, green!50!black, -latex] (vert310.west)--(vert210.east);
 \draw[style=semithick, green!50!black, -latex] (vert300.west)--(vert200.east);
 
 \draw[dotted, style=semithick, blue, -latex] (vert020.south)--(vert010.north);
 \draw[dotted, style=semithick, blue, -latex] (vert120.south)--(vert110.north);
 \draw[dotted, style=semithick, blue, -latex] (vert220.south)--(vert210.north);
 \draw[dotted, style=semithick, blue, -latex] (vert320.south)--(vert310.north);
 \draw[dotted, style=semithick, blue, -latex] (vert010.south)--(vert000.north);
 \draw[dotted, style=semithick, blue, -latex] (vert110.south)--(vert100.north);
 \draw[dotted, style=semithick, blue, -latex] (vert210.south)--(vert200.north);
 \draw[dotted, style=semithick, blue, -latex] (vert310.south)--(vert300.north);
 
 \draw[dotted, style=semithick, blue, -latex] (vert021.south)--(vert011.north);
 \draw[dotted, style=semithick, blue, -latex] (vert121.south)--(vert111.north);
 \draw[dotted, style=semithick, blue, -latex] (vert221.south)--(vert211.north);
 \draw[dotted, style=semithick, blue, -latex] (vert321.south)--(vert311.north);
 \draw[dotted, style=semithick, blue, -latex] (vert011.south)--(vert001.north);
 \draw[dotted, style=semithick, blue, -latex] (vert111.south)--(vert101.north);
 \draw[dotted, style=semithick, blue, -latex] (vert211.south)--(vert201.north);
 \draw[dotted, style=semithick, blue, -latex] (vert311.south)--(vert301.north);

 \draw[style=semithick, green!50!black, -latex] (vert121.west)--(vert021.east);
 \draw[style=semithick, green!50!black, -latex] (vert111.west)--(vert011.east);
 \draw[style=semithick, green!50!black, -latex] (vert101.west)--(vert001.east);
 \draw[style=semithick, green!50!black, -latex] (vert221.west)--(vert121.east);
 \draw[style=semithick, green!50!black, -latex] (vert211.west)--(vert111.east);
 \draw[style=semithick, green!50!black, -latex] (vert201.west)--(vert101.east);
 \draw[style=semithick, green!50!black, -latex] (vert321.west)--(vert221.east);
 \draw[style=semithick, green!50!black, -latex] (vert311.west)--(vert211.east);
 \draw[style=semithick, green!50!black, -latex] (vert301.west)--(vert201.east);
 
 \draw[dashed, style=semithick, red, -latex] (vert000)--(vert001);
 \draw[dashed, style=semithick, red, -latex] (vert100)--(vert101);
 \draw[dashed, style=semithick, red, -latex] (vert200)--(vert201);
 \draw[dashed, style=semithick, red, -latex] (vert300)--(vert301);
 \draw[dashed, style=semithick, red, -latex] (vert010)--(vert011);
 \draw[dashed, style=semithick, red, -latex] (vert110)--(vert111);
 \draw[dashed, style=semithick, red, -latex] (vert210)--(vert211);
 \draw[dashed, style=semithick, red, -latex] (vert310)--(vert311);
 \draw[dashed, style=semithick, red, -latex] (vert020)--(vert021);
 \draw[dashed, style=semithick, red, -latex] (vert120)--(vert121);
 \draw[dashed, style=semithick, red, -latex] (vert220)--(vert221);
 \draw[dashed, style=semithick, red, -latex] (vert320)--(vert321);

\end{tikzpicture}

\vspace{.2in}
\label{grid}
}
\hspace{.5in}
\subfigure[A 1-skeleton of a 3-graph requiring factorization rules]{
\begin{tikzpicture}[scale=1.25]

 \node[inner sep=.5pt, circle, fill=black] (vertA) at (-1,0) [draw] {};%
        \node[inner  sep=4pt, anchor = south] at (vertA.north) {$\scriptstyle v$};%
        \node[inner sep=.5pt, circle, fill=black] (vertB) at (1,0) [draw] {};%
        \node[inner sep=4pt, anchor = south] at (vertB.north) {$\scriptstyle w$};%
        \draw[dotted, style=semithick, blue, -latex] (vertA.north east)%
            parabola[parabola height=0.5cm] (vertB.north west);%
        \node[inner sep=2pt, anchor = south] at (0,0.5) {$\scriptstyle{\color{blue}a}$};%
        \draw[dotted, style=semithick, blue, -latex] (vertB.south west)%
            parabola[parabola height=-0.5cm] (vertA.south east);%
        \node[inner sep=2pt, anchor = north] at (0,-0.5) {$\scriptstyle{\color{blue}b}$};%
        \draw[dotted, style=semithick, blue, -latex] (vertA.north east)%
            parabola[parabola height=0.85cm] (vertB.north west);%
        \node[inner sep=2pt, anchor = south] at (0,0.85) {$\scriptstyle{\color{blue}a'}$};%
        \draw[dotted, style=semithick, blue, -latex] (vertB.south west)%
            parabola[parabola height=-0.85cm] (vertA.south east);%
        \node[inner sep=2pt, anchor = north] at (0,-0.85) {$\scriptstyle{\color{blue}b'}$};%
        \draw[dashed, style=semithick, red, -latex] (vertA.north west) .. %
            controls (-1.25,0.25) and (-1.5,0.25) .. (-1.5,0) .. %
            controls (-1.5,-0.25) and (-1.25,-0.25) .. (vertA.south west);%
        \node[inner sep=2pt, anchor = east] at (-1.5,0) {$\scriptstyle{\color{red}c}$};%
        \draw[dashed, style=semithick, red, -latex]
        (vertB.north east) .. controls (1.25,0.25) and (1.5,0.25) .. (1.5,0) .. %
            controls (1.5,-0.25) and (1.25,-0.25) .. (vertB.south east);%
        \node[inner sep=2pt, anchor = west] at (1.5,0) {$\scriptstyle{\color{red}d}$};%
        \draw[style=semithick, green!50!black, -latex] (vertA.north west) .. %
            controls (-1.25,0.75) and (-2,0.5) .. (-2,0) .. %
            controls (-2,-0.5) and (-1.25,-0.75) .. (vertA.south west);%
        \node[inner sep=2pt, anchor = east] at (-2,0) {$\scriptstyle{\color{green!50!black}e}$};%
        \draw[style=semithick, green!50!black, -latex] (vertB.north east) .. %
            controls (1.25,0.75) and (2,0.5) .. (2,0) .. %
            controls (2,-0.5) and (1.25,-0.75) .. (vertB.south east);%
        \node[inner sep=2pt, anchor = west] at (2,0) {$\scriptstyle{\color{green!50!black}f}$};%


\end{tikzpicture}
\label{circles}
}
\caption{1-Skeletons of two different 3-graphs}
\end{figure}
In some cases, such as any $\Okm$, we can construct the entire $k$-graph from the 1-skeleton.  In other cases we need more information.  Consider the path $bd$ in Figure~\ref{circles} of shape $(0,1,1)$.  By the unique factorization property, $bd$ must factor uniquely as a ``red-blue" path.  So, $bd = \mu\nu$ with $d(\mu) = (0,0,1)$ and $d(\nu) = (0,1,0)$.  It cannot be determined from the 1-skeleton alone if the unique factorization we desire is $cb$ or $cb'$.  We need to specify factorization rules for each bi-colored path:
$$da = a'c \hspace{.75in}fa = a'e\hspace{.75in}bd = cb'$$$$bf = eb'\hspace{.75in}fd = df \hspace{.75in}ce = ec.
$$

For topological $k$-graphs, $\left(\L_1, d_1\right)$ and $\left(\L_2, d_2\right)$, a \emph{topological $k$-graph morphism} is a continuous degree preserving functor $x:\L_1 \rightarrow \L_2$.  That is, $d_2(x(\l)) = d_1(\l)$.  We are particularly concerned with the graph morphisms from the grid graphs.  The \emph{path space} of a topological $k$-graph $\L$ $$X_\L := \bigcup_{m \in \left(\NN \cup\{\infty\}\right)^k} \left\{x : \Omega_{k,m} \rightarrow \L \, |\, x \text{ is a graph morphism}\right\}$$ can be thought to include $\L$ since any finite path $\l$ uniquely determines a graph morphism $x_\l :\Omega_{k,d(\l)} \rightarrow \L$ such that $x(m,n) = \l(m,n)$ for any $m \leq n \leq d(\l)$.  We extend the idea of range and degree so that $r(x) = x(0,0)$ and $d(x) = m$ for $x: \Omega_{k,m} \rightarrow \L$.  In practice, we think of taking a particular grid graph and labeling the vertices and edges while following the structure of and factorization rules associated to the 1-skeleton of $\L$.  So, a path of shape $(\infty, \infty, \infty)$ with range $v$ in the 3-graph of Figure~\ref{circles} with the factorization rules given could look like that of Figure~\ref{infinite path}.

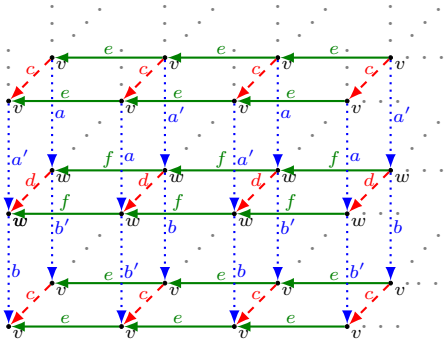
\begin{figure}[h]
\begin{center}

\begin{tikzpicture}[scale = 1.5]

 \node[inner sep=.5pt, circle, fill=black] (vert000) at (0,0,0) [draw] {};
 \node[inner sep=.5pt, circle, fill=black] (vert100) at (1,0,0) [draw] {};
 \node[inner sep=.5pt, circle, fill=black] (vert200) at (2,0,0) [draw] {};
 \node[inner sep=.5pt, circle, fill=black] (vert300) at (3,0,0) [draw] {};
 
 \node[inner sep=.5pt, circle, fill=black] (vert001) at (0,0,1) [draw] {};
 \node[inner sep=.5pt, circle, fill=black] (vert101) at (1,0,1) [draw] {};
 \node[inner sep=.5pt, circle, fill=black] (vert201) at (2,0,1) [draw] {};
 \node[inner sep=.5pt, circle, fill=black] (vert301) at (3,0,1) [draw] {};
 
 \node[inner sep=.5pt, circle, fill=black] (vert010) at (0,1,0) [draw] {};
 \node[inner sep=.5pt, circle, fill=black] (vert110) at (1,1,0) [draw] {};
 \node[inner sep=.5pt, circle, fill=black] (vert210) at (2,1,0) [draw] {};
 \node[inner sep=.5pt, circle, fill=black] (vert310) at (3,1,0) [draw] {};
 
 \node[inner sep=.5pt, circle, fill=black] (vert011) at (0,1,1) [draw] {};
 \node[inner sep=.5pt, circle, fill=black] (vert111) at (1,1,1) [draw] {};
 \node[inner sep=.5pt, circle, fill=black] (vert211) at (2,1,1) [draw] {};
 \node[inner sep=.5pt, circle, fill=black] (vert311) at (3,1,1) [draw] {}; 
 
 \node[inner sep=.5pt, circle, fill=black] (vert020) at (0,2,0) [draw] {};
 \node[inner sep=.5pt, circle, fill=black] (vert120) at (1,2,0) [draw] {};
 \node[inner sep=.5pt, circle, fill=black] (vert220) at (2,2,0) [draw] {};
 \node[inner sep=.5pt, circle, fill=black] (vert320) at (3,2,0) [draw] {};
 
 \node[inner sep=.5pt, circle, fill=black] (vert021) at (0,2,1) [draw] {};
 \node[inner sep=.5pt, circle, fill=black] (vert121) at (1,2,1) [draw] {};
 \node[inner sep=.5pt, circle, fill=black] (vert221) at (2,2,1) [draw] {};
 \node[inner sep=.5pt, circle, fill=black] (vert321) at (3,2,1) [draw] {};

 \draw[style=thick, green!50!black, -latex] (vert120.west)--(vert020.east);
 \draw[style=thick, green!50!black, -latex] (vert110.west)--(vert010.east);
 \draw[style=thick, green!50!black, -latex] (vert100.west)--(vert000.east);
 \draw[style=thick, green!50!black, -latex] (vert220.west)--(vert120.east);
 \draw[style=thick, green!50!black, -latex] (vert210.west)--(vert110.east);
 \draw[style=thick, green!50!black, -latex] (vert200.west)--(vert100.east);
 \draw[style=thick, green!50!black, -latex] (vert320.west)--(vert220.east);
 \draw[style=thick, green!50!black, -latex] (vert310.west)--(vert210.east);
 \draw[style=thick, green!50!black, -latex] (vert300.west)--(vert200.east);
 
 \draw[dotted, style=thick, blue, -latex] (vert020.south)--(vert010.north);
 \draw[dotted, style=thick, blue, -latex] (vert120.south)--(vert110.north);
 \draw[dotted, style=thick, blue, -latex] (vert220.south)--(vert210.north);
 \draw[dotted, style=thick, blue, -latex] (vert320.south)--(vert310.north);
 \draw[dotted, style=thick, blue, -latex] (vert010.south)--(vert000.north);
 \draw[dotted, style=thick, blue, -latex] (vert110.south)--(vert100.north);
 \draw[dotted, style=thick, blue, -latex] (vert210.south)--(vert200.north);
 \draw[dotted, style=thick, blue, -latex] (vert310.south)--(vert300.north);
 
 \draw[dotted, style=thick, blue, -latex] (vert021.south)--(vert011.north);
 \draw[dotted, style=thick, blue, -latex] (vert121.south)--(vert111.north);
 \draw[dotted, style=thick, blue, -latex] (vert221.south)--(vert211.north);
 \draw[dotted, style=thick, blue, -latex] (vert321.south)--(vert311.north);
 \draw[dotted, style=thick, blue, -latex] (vert011.south)--(vert001.north);
 \draw[dotted, style=thick, blue, -latex] (vert111.south)--(vert101.north);
 \draw[dotted, style=thick, blue, -latex] (vert211.south)--(vert201.north);
 \draw[dotted, style=thick, blue, -latex] (vert311.south)--(vert301.north);

 \draw[style=thick, green!50!black, -latex] (vert121.west)--(vert021.east);
 \draw[style=thick, green!50!black, -latex] (vert111.west)--(vert011.east);
 \draw[style=thick, green!50!black, -latex] (vert101.west)--(vert001.east);
 \draw[style=thick, green!50!black, -latex] (vert221.west)--(vert121.east);
 \draw[style=thick, green!50!black, -latex] (vert211.west)--(vert111.east);
 \draw[style=thick, green!50!black, -latex] (vert201.west)--(vert101.east);
 \draw[style=thick, green!50!black, -latex] (vert321.west)--(vert221.east);
 \draw[style=thick, green!50!black, -latex] (vert311.west)--(vert211.east);
 \draw[style=thick, green!50!black, -latex] (vert301.west)--(vert201.east);
 
 \draw[dashed, style=thick, red, -latex] (vert000)--(vert001);
 \draw[dashed, style=thick, red, -latex] (vert100)--(vert101);
 \draw[dashed, style=thick, red, -latex] (vert200)--(vert201);
 \draw[dashed, style=thick, red, -latex] (vert300)--(vert301);
 \draw[dashed, style=thick, red, -latex] (vert010)--(vert011);
 \draw[dashed, style=thick, red, -latex] (vert110)--(vert111);
 \draw[dashed, style=thick, red, -latex] (vert210)--(vert211);
 \draw[dashed, style=thick, red, -latex] (vert310)--(vert311);
 \draw[dashed, style=thick, red, -latex] (vert020)--(vert021);
 \draw[dashed, style=thick, red, -latex] (vert120)--(vert121);
 \draw[dashed, style=thick, red, -latex] (vert220)--(vert221);
 \draw[dashed, style=thick, red, -latex] (vert320)--(vert321);

\IoIIIdimdots(3.15/3.25/3.4,0,0)
            \IoIIIdimdots(3.15/3.3/3.45,1,0)
            \IoIIIdimdots(3.15/3.3/3.45,2,0)
            \IoIIIdimdots(3.15/3.3/3.45,0,1)
            \IoIIIdimdots(3.15/3.3/3.45,1,1)
            \IoIIIdimdots(3.15/3.3/3.45,2,1)
            \IIoIIIdimdots(0,2.15/2.3/2.45,0)
            \IIoIIIdimdots(1,2.15/2.3/2.45,0)
            \IIoIIIdimdots(2,2.15/2.3/2.45,0)
            \IIoIIIdimdots(3,2.15/2.3/2.45,0)
            \IIoIIIdimdots(0,2.15/2.3/2.45,1)
            \IIoIIIdimdots(1,2.15/2.3/2.45,1)
            \IIoIIIdimdots(2,2.15/2.3/2.45,1)
            \IIoIIIdimdots(3,2.15/2.3/2.45,1)
            \IIIoIIIdimdots(0,0,-0.5/-0.8/-1.1)
            \IIIoIIIdimdots(1,0,-0.5/-0.8/-1.1)
            \IIIoIIIdimdots(2,0,-0.5/-0.8/-1.1)
            \IIIoIIIdimdots(3,0,-0.5/-0.8/-1.1)
            \IIIoIIIdimdots(0,1,-0.5/-0.8/-1.1)
            \IIIoIIIdimdots(1,1,-0.5/-0.8/-1.1)
            \IIIoIIIdimdots(2,1,-0.5/-0.8/-1.1)
            \IIIoIIIdimdots(3,1,-0.5/-0.8/-1.1)
            \IIIoIIIdimdots(0,2,-0.5/-0.8/-1.1)
            \IIIoIIIdimdots(1,2,-0.5/-0.8/-1.1)
            \IIIoIIIdimdots(2,2,-0.5/-0.8/-1.1)
            \IIIoIIIdimdots(3,2,-0.5/-0.8/-1.1)

     \node[inner sep = 1.25pt, anchor = north west] at (0, 0, 0) {\tiny{$v$}};
          \node[inner sep = 1.25pt, anchor = north west] at (0, 0, 1) {\tiny{$v$}};
     \node[inner sep = 1.25pt, anchor = north west] at (1, 0, 1) {\tiny{$v$}};
     \node[inner sep = 1.25pt, anchor = north west] at (0, 1, 1) {\tiny{$w$}};
     
     \node[inner sep = 1.25pt, anchor = north west] at (1, 0, 0) {\tiny{$v$}};
     \node[inner sep = 1.25pt, anchor = north west] at (2, 0, 0) {\tiny{$v$}};
     \node[inner sep = 1.25pt, anchor = north west] at (2, 0, 1) {\tiny{$v$}};
     \node[inner sep = 1.25pt, anchor = north west] at (3, 0, 0) {\tiny{$v$}};
     \node[inner sep = 1.25pt, anchor = north west] at (3, 0, 1) {\tiny{$v$}};     
     \node[inner sep = 1.25pt, anchor = north west] at (2, 2, 0) {\tiny{$v$}};
     \node[inner sep = 1.25pt, anchor = north west] at (2, 2, 1) {\tiny{$v$}};
     \node[inner sep = 1.25pt, anchor = north west] at (3, 2, 0) {\tiny{$v$}};
     \node[inner sep = 1.25pt, anchor = north west] at (3, 2, 1) {\tiny{$v$}};
     \node[inner sep = 1.25pt, anchor = north west] at (0, 2, 0) {\tiny{$v$}};
     \node[inner sep = 1.25pt, anchor = north west] at (0, 2, 1) {\tiny{$v$}};
     \node[inner sep = 1.25pt, anchor = north west] at (1, 2, 0) {\tiny{$v$}};
     \node[inner sep = 1.25pt, anchor = north west] at (1, 2, 1) {\tiny{$v$}};
     
     \node[inner sep = 1.25pt, anchor = north west] at (2, 1, 0) {\tiny{$w$}};
     \node[inner sep = 1.25pt, anchor = north west] at (2, 1, 1) {\tiny{$w$}};
     \node[inner sep = 1.25pt, anchor = north west] at (3, 1, 0) {\tiny{$w$}};
     \node[inner sep = 1.25pt, anchor = north west] at (3, 1, 1) {\tiny{$w$}};
     \node[inner sep = 1.25pt, anchor = north west] at (0, 1, 0) {\tiny{$w$}};
     \node[inner sep = 1.25pt, anchor = north west] at (0, 1, 1) {\tiny{$w$}};
     \node[inner sep = 1.25pt, anchor = north west] at (1, 1, 0) {\tiny{$w$}};
     \node[inner sep = 1.25pt, anchor = north west] at (1, 1, 1) {\tiny{$w$}};
     
     \node[inner sep=.75pt, anchor= west] at (0, .5, 1) {\tiny{\color{blue} $b$}};
     \node[inner sep=.75pt, anchor= west] at (0, .5, 0) {\tiny{\color{blue} $b'$}};
     \node[inner sep=.75pt, anchor= west] at (1, .5, 1) {\tiny{\color{blue} $b'$}};
     \node[inner sep=.75pt, anchor= west] at (1, .5, 0) {\tiny{\color{blue} $b$}};
     \node[inner sep=.75pt, anchor= west] at (2, .5, 1) {\tiny{\color{blue} $b$}};
     \node[inner sep=.75pt, anchor= west] at (2, .5, 0) {\tiny{\color{blue} $b'$}};
     \node[inner sep=.75pt, anchor= west] at (3, .5, 1) {\tiny{\color{blue} $b'$}};
     \node[inner sep=.75pt, anchor= west] at (3, .5, 0) {\tiny{\color{blue} $b$}};
     
     \node[inner sep=.75pt, anchor= west] at (0, 1.5, 1) {\tiny{\color{blue} $a'$}};
     \node[inner sep=.75pt, anchor= west] at (0, 1.5, 0) {\tiny{\color{blue} $a$}};
     \node[inner sep=.75pt, anchor= west] at (1, 1.5, 1) {\tiny{\color{blue} $a$}};
     \node[inner sep=.75pt, anchor= west] at (1, 1.5, 0) {\tiny{\color{blue} $a'$}};
     \node[inner sep=.75pt, anchor= west] at (2, 1.5, 1) {\tiny{\color{blue} $a'$}};
     \node[inner sep=.75pt, anchor= west] at (2, 1.5, 0) {\tiny{\color{blue} $a$}};
     \node[inner sep=.75pt, anchor= west] at (3, 1.5, 1) {\tiny{\color{blue} $a$}};
     \node[inner sep=.75pt, anchor= west] at (3, 1.5, 0) {\tiny{\color{blue} $a'$}};
     
     \node[inner sep=1pt, anchor= south] at (.5, 0, 1) {\textbf{\tiny{\color{green!50!black} $e$}}};
     \node[inner sep=1pt, anchor= south] at (.5, 0, 0) {\textbf{\tiny{\color{green!50!black} $e$}}};
     \node[inner sep=1pt, anchor= south] at (1.5, 0, 1) {\textbf{\tiny{\color{green!50!black} $e$}}};
     \node[inner sep=1pt, anchor= south] at (1.5, 0, 0) {\textbf{\tiny{\color{green!50!black} $e$}}};                 
     \node[inner sep=1pt, anchor= south] at (2.5, 0, 1) {\textbf{\tiny{\color{green!50!black} $e$}}};
     \node[inner sep=1pt, anchor= south] at (2.5, 0, 0) {\textbf{\tiny{\color{green!50!black} $e$}}};    
     \node[inner sep=1pt, anchor= south] at (.5, 2, 1) {\textbf{\tiny{\color{green!50!black} $e$}}};
     \node[inner sep=1pt, anchor= south] at (.5, 2, 0) {\textbf{\tiny{\color{green!50!black} $e$}}};
     \node[inner sep=1pt, anchor= south] at (1.5, 2, 1) {\textbf{\tiny{\color{green!50!black} $e$}}};
     \node[inner sep=1pt, anchor= south] at (1.5, 2, 0) {\textbf{\tiny{\color{green!50!black} $e$}}};                 
     \node[inner sep=1pt, anchor= south] at (2.5, 2, 1) {\textbf{\tiny{\color{green!50!black} $e$}}};
     \node[inner sep=1pt, anchor= south] at (2.5, 2, 0) {\textbf{\tiny{\color{green!50!black} $e$}}};    
     
     \node[inner sep=1pt, anchor= south] at (.5, 1, 1) {\textbf{\tiny{\color{green!50!black} $f$}}};
     \node[inner sep=1pt, anchor= south] at (.5, 1, 0) {\textbf{\tiny{\color{green!50!black} $f$}}};
     \node[inner sep=1pt, anchor= south] at (1.5, 1, 1) {\textbf{\tiny{\color{green!50!black} $f$}}};
     \node[inner sep=1pt, anchor= south] at (1.5, 1, 0) {\textbf{\tiny{\color{green!50!black} $f$}}};                 
     \node[inner sep=1pt, anchor= south] at (2.5, 1, 1) {\textbf{\tiny{\color{green!50!black} $f$}}};
     \node[inner sep=1pt, anchor= south] at (2.5, 1, 0) {\textbf{\tiny{\color{green!50!black} $f$}}};    
     
     \node[inner sep=2pt, anchor= south] at (0, 0 , .5) {\tiny{\color{red} $c$}};
     \node[inner sep=2pt, anchor= south] at (1, 0 , .5) {\tiny{\color{red} $c$}};
     \node[inner sep=2pt, anchor= south] at (2, 0 , .5) {\tiny{\color{red} $c$}};
     \node[inner sep=2pt, anchor= south] at (3, 0 , .5) {\tiny{\color{red} $c$}};
     
     \node[inner sep=2pt, anchor= south] at (0, 1 , .5) {\tiny{\color{red} $d$}};
     \node[inner sep=2pt, anchor= south] at (1, 1 , .5) {\tiny{\color{red} $d$}};
     \node[inner sep=2pt, anchor= south] at (2, 1 , .5) {\tiny{\color{red} $d$}};
     \node[inner sep=2pt, anchor= south] at (3, 1 , .5) {\tiny{\color{red} $d$}};
     
     \node[inner sep=2pt, anchor= south] at (0, 2 , .5) {\tiny{\color{red} $c$}};
     \node[inner sep=2pt, anchor= south] at (1, 2 , .5) {\tiny{\color{red} $c$}};
     \node[inner sep=2pt, anchor= south] at (2, 2 , .5) {\tiny{\color{red} $c$}};
     \node[inner sep=2pt, anchor= south] at (3, 2 , .5) {\tiny{\color{red} $c$}};

\end{tikzpicture}

\end{center}
\caption{An Infinite Path in a 3-Graph}
\label{infinite path}
\end{figure}

In order to define aperiodic paths in a topological $k$-graph we need the idea of the shift map.  The \emph{shift map} $\sigma^m$ on $X_\L$ assigns to $x$ the unique path $\sigma^mx$ in $X_\L$ which satisfies $d(\sigma^mx) = d(x) - m$ and $\sigma^mx(0, n) = x(m, m+n)$ for all $0 \leqslant n \leqslant d(x) - m$.  For $m \in \NN^k$, the shift map of degree $m$ removes a path of shape $m$ from the range end of $x$.  We say a path $x \in X_\L$ is \emph{aperiodic} if for any $m, n \in \NN^k$ $\sigma^mx = \sigma^nx$ implies that $m = n$.  The path of Figure~\ref{infinite path} is periodic, as $\sigma^{(2, 0,1)}x = x$.

The idea of a minimal common extension of two paths will be important in one of our formulations of aperiodicity and it is also necessary in our definition of a boundary path.  For paths $\mu, \nu \in \L$, we say $\l$ is a \emph{common extension of} $\mu$ and $\nu$ if we can factor $\l = \mu\mu'$ and $\l = \nu\nu'$ for some $\mu', \nu' \in \L$.  We consider $\l$ to be a \emph{minimal common extension} if it also satisfies $d(\l) = d(\mu) \vee d(\nu)$.  We denote by $\MCE(\mu, \nu)$ the set of minimal common extensions of $\mu$ and $\nu$ and for subsets $X, Y \subset \L$ define $$\MCE(X,Y) := \bigcup_{\mu\in X, \nu \in Y} \MCE(\mu, \nu).$$A topological $k$-graph $\L$ is \emph{compactly aligned} if whenever $X,Y \subset \L$ are compact, then $\MCE(X,Y)$ is also compact.

For a vertex $v \in \L^0$, we say a subset $E \subset \L$ is \emph{compact exhaustive for $v$} if $E$ is compact, $r(E)$ is a neighborhood of $v$, and for all $\l \in r(E)\L$ there exists a $\mu \in E$ such that $\MCE(\l,\mu) \neq \emptyset$.  A path $x \in X_\L$ is a \emph{boundary path} if for any $m \in \NN^k$ with $m \leq d(x)$ and any set $E$ which is compact exhaustive set for $x(m,m)$, there exists a $\l \in E$ such that $x(m, m + d(\l)) = \l$.  We write $\partial\L$ for the set of all boundary paths in $X_\L$.  As in discrete $k$-graphs the boundary paths are a generalization of the infinite paths of directed graphs, \cite[Lemma 4.22]{MyThesis}.

Under the condition that $\L$ is compactly aligned, we can build a topological groupoid $\mathcal{G}_\L$ called the \emph{boundary path groupoid} which has $\partial\L$ as its unit space and can be endowed with a locally compact Hausdorff topology.  The C$^*$-algebra of $\L$, $C^*(\L)$, is the full groupoid C$^*$-algebra of the boundary path groupoid, $C^*(\mathcal{G}_\L)$.  As we will not need the details of the construction of the C$^*$-algebra to understand the results of this paper, we refer the reader to \cite{Yeend1} for the formulation of this topology as well as a more in depth discussion.  However, to demonstrate the importance of aperiodicity, we conclude the background information with a result of Yeend.

\begin{prop}\cite[Theorem 5.2]{Yeend1} For a compactly aligned topological $k$-graph $(\L, d)$, the boundary path groupoid $\mathcal{G}_\L$ is topologically principal if and only if for every open set $V \subset \L^0$ there exists an $x \in V\partial\L$ which is aperiodic. \end{prop}

See \cite{RSWY}, for structure and simplicity arguments relying on the groupoid being topologically principal and amenable.  The authors use groupoid techniques to prove the equivalence of \eqref{Y} and \eqref{W} of the main theorem of this paper.  Theorem~\ref{BigTheorem} remains the cleaner and more accessible proof of the equivalence of the two notions of aperiodicity

\section{Results}

The condition that every open set of vertices must have an aperiodic path that terminates in that set is the natural extension of aperiodicity conditions for topological graphs and discrete $k$-graphs.  Yeend introduced the condition and referred to it as Condition (A).  Our main result, Theorem~\ref{BigTheorem}, gives two new equivalent conditions to Yeend's Condition (A).

\begin{thm}\label{BigTheorem}

Let $(\L, d)$ be a compactly aligned topological $k$-graph with no sources.  Then the following conditions are equivalent.
	
	\begin{enumerate}[(A)]
	
		\item  \label{Y}For any open set $V \subset \L^0$ there exists an aperiodic path $x \in V \partial\L$.
		\item  \label{F}For any open set $V \subset \L^0$ and any pair $m \neq n \in \NN^k$ there exists a path $\l_{V, m, n} \in V\L$ such that $d(\l) \geq m \vee n$ and 
		\begin{equation}\label{star}\tag{$\star$}
			\l(m, m + d(\l) - (m \vee n)) \neq \l(n, n + d(\l) - (m \vee n)).
		\end{equation}
		\item  \label{W}For any distinct pair of open sets $X, Y \subset \L$ such that $s(X) = s(Y)$ and both $s|_X$ and $s|_Y$ are homeomorphisms there exists $\tau \in s(X)\L$ such that $\MCE(X\tau, Y\tau) = \emptyset$.
		
	\end{enumerate}

\end{thm}

It is not difficult to see that in the discrete case, (if $\Obj(\L)$ and $\Mor(\L)$ both have the discrete topology) then Conditions \eqref{Y} -- \eqref{W} above reduce to \cite[Definition 4.3]{KP}, Condition (iv) of \cite[Lemma 3.2]{RandS}, and \cite[Definition 3.1]{LS} respectively.  This special case gives the following corollary.  The equivalence that appears in Corollary~\ref{discrete} is already known, as each condition is a necessary condition for simplicity of the associated C$^*$-algebra.  The direct proof of this equivalence that follows from the proof of Theorem~\ref{BigTheorem} is new, and an important outcome of this work.

\begin{cor}\label{discrete}

Let $(\L, d)$ be a discrete finitely aligned $k$-graph with no sources.  Then the following are equivalent.

	\begin{enumerate}[(i)]
	
	\item  For each $v \in \L^0$ there exists and aperiodic path $x \in v\partial\L$.
	\item  For each $v \in \L^0$ and each $m \neq n \in \NN^k$ there exists a path $\l \in v\L$ such that $d(\l) \geq m \vee n$ and which satisfies \eqref{star}.
	\item   For every pair of distinct paths $\alpha, \beta \in \L$ with $s(\alpha) = s(\beta)$ there exists a path $\tau \in s(\alpha)\L$ such that $\MCE(\alpha\tau, \beta\tau) = \emptyset$.

	\end{enumerate}

\end{cor}

The proof of the theorem requires an important lemma.  Lemma \ref{TubeLemma} is an extension of Lemma 5.6 of \cite{Katsura1} which shows that near a path which satisfies \eqref{star} are a lot of other paths which also satisfy \eqref{star}.  This allows us to extend the proof of the equivalence of (i) and (iv) in Lemma 3.2 of \cite{RandS} to topological $k$-graphs.

\begin{lemma}\label{TubeLemma}

Let $V$ be a nonempty open subset of $\L^0$, $m \neq n \in \NN^k$, and $\l \in V\L$.  If $\l$ satisfies \eqref{star} for $m$ and $n$, then there exists a compact neighborhood $E \subset V\L^{d(\l)}$of $\l$ such that every $\mu \in E$ satisfies \eqref{star}.

\end{lemma}

	\begin{proof}
	
	Suppose $\l \in V\L$ satisfies \eqref{star}.  We can factor $\l$ in two ways:
\begin{center}
$\l = \l(0, m)\l(m, m + d(\l) - (m \vee n))\l(m + d(\l) - (m \vee n), d(\l))$ and\\
$\l = \l(0, n)\l(n, n + d(\l) - (m \vee n))\l(n + d(\l) - (m \vee n)), d(\l))$
\end{center}
Let $E_m$ and $E_n$ be disjoint compact neighborhoods of $\l(m, m+ d(\l) - (m \vee n))$ and $\l(n, n+ d(\l) - (m \vee n))$.  Also, let $E_1$, $E_2$, $E_3$, and $E_4$ be compact neighborhoods of $\l(0, m)$, $\l(m + d(\l) - (m \vee n), d(\l))$, $\l(0, n)$, and $\l(n + d(\l) - (m \vee n), d(\l))$ respectively.  We define the set $E'$ to be the paths in $\L^{d(\l)}$ that can be factored as paths in $E_1$, $E_m$ and $E_2$ as well as $E_3$, $E_n$, and $E_4$.  So, \begin{center}$E' := \{ \alpha\beta\gamma : (\alpha, \beta, \gamma) \in E_1 \times_c E_m \times_c E_2\} \cap \{\mu\nu\xi : (\mu, \nu, \xi) \in E_3 \times_c E_n \times_c E_4\} \subset \L^{d(\l)}$.\end{center}
Now, let $F \subset V$ be a compact neighborhood of $r(\l)$ and \begin{center}$ E = E' \cap r_{d(\l)}^{-1}(F)$.\end{center}
Then $E \subset V\L^{d(\l)}$ is a compact neighborhood of $\l$ with nonempty interior in which every element satisfies \eqref{star}.
	\end{proof}

\begin{proof}[Proof of Theorem~\ref{BigTheorem}]

We will show Condition \eqref{F} is equivalent to each of \eqref{Y} and \eqref{W}.

\eqref{Y} $\Longrightarrow$ \eqref{F}.  Fix and open set $V \subset \L^0$, and a pair $m \neq n \in \NN^k$, and suppose $x \in V\partial\L$ is aperiodic.  Then, $\sigma^mx \neq \sigma^nx$ and so for a sufficiently large $p \in \NN^k$ $$\sigma^mx(0,p) \neq \sigma^nx(0,p).$$  Let $\l = x(0, p + m \vee n)$, then
\begin{align*}
\l(m, m + d(\l) - m \vee n) & = x(m, m + (p + m \vee n) - m \vee n)\\
& = \sigma^mx(0, p)\\
& \neq \sigma^nx(0,p)\\
& = x(n, n + p + m \vee n - m \vee n)\\
& = \l(n, n + d(\l) - m \vee n).
\end{align*}
So, $d(\l) \geq m \vee n$ and $\l$ satisfies \eqref{star}.

\eqref{F} $\Longrightarrow$ \eqref{Y}.  Fix an open set $V \subseteq \L^0$ and let $\{(m_i, n_i)\}^{\infty}_{i = 1}$ be a listing of the elements in the set $\left\{(m, n) \in \NN^k \times \NN^k : m \neq n\right\}$.
Let $V_1 = V$, choose $\l_1 \in V_1\L$ such that $d(\l_1) \geq m_1 \vee n_1$ and \eqref{star} is satisfied for $m_1$ and $n_1$; choose a compact neighborhood $F_1$ of $\l_1$ by Lemma~\ref{TubeLemma}; and let $E_1 := F_1\partial\L = \left\{x \in \partial \L : x\left(0, d(\l_1)\right) \in F_1\right\}$.
Now, proceed inductively:
\begin{align*}
V_i & := \text{interior of } s(F_{i-1}),\\
\l_i & := \l_{V_i, m_i, n_i} \text{ satisfy \eqref{star} for $m_i$ and $n_i$},\\
F_i & := \text{a compact neighborhood of } \l_i \text{ given by Lemma~\ref{TubeLemma}, and}\\
E_i & := F_1\dots F_i\partial\L.
\end{align*}
Notice that each $E_i$ is compact in $\partial\L$ by Lemma 3.8 of \cite{Yeend2}:  Let $p = q = \sum_{j=1}^i d(\l_j)$ and $U = V = F_1\dots F_i \subset \L^{\sum_{j=1}^i d(\l_j)}$.  Then $E_i$ is exactly the set $\mathcal{Z}(U \ast_s V, p-q)$, which is compact in $\partial\L$.  Also note that the $E_i$'s are nested; $E_1 \supset E_2 \supset \dots$.  Now, let $$E := \bigcap_{i=1}^{\infty} E_i \subset \partial\L.$$
The set $E$ is non-empty, so take any $x \in E$.
Let $\mu_i \! = \! x\left(d(\l_1\l_2 \dots \l_{i-1}), d(\l_i)\right)$ for each $i \in \NN$.  Then by Lemma 4.1.3 of \cite{YeendThesis}, $x$ is the unique element in $X_\L$ such that $d(x) = \lim_{j \rightarrow \infty} d(\mu_1\mu_2\dots\mu_j)$ and $x(0, d(\mu_1\mu_2\dots\mu_j)) = \mu_1\mu_2\dots\mu_j$, so it makes sense to write$$x = \mu_1\mu_2\mu_3\dots, $$ with each $\mu_i \in F_i$.
Fix $m \neq n \in \NN^k$.   Then for some $i \in \NN$, $(m, n) = (m_i, n_i)$ in the listing above.  So,
\begin{align*}
\sigma^mx\!\!\left(\sum_{j=1}^{i-1} d(\l_j), \sum_{j=1}^i d(\l_j) - (m \vee n)\!\!\right)\!\! & = \sigma^{m + \sum_{j=1}^{i-1} d(\l_j)} x(0, d(\l_i) - (m \vee n))\\
& = \sigma^{\sum_{j=1}^{i-1}d(\l_j)}x(m, m+ d(\l_i) \! - \! (m \vee n))\\
& = \mu_i(m, m+ d(\mu_i) - (m \vee n)).
\end{align*}
Similarly,
$$\sigma^nx\left(\sum_{j=1}^{i-1} d(\l_j), \sum_{j=1}^i d(\l_j) - (m \vee n)\right) = \mu_i(n, n + d(\mu_i) - (m \vee n)).$$
By the definition of $F_i$, we know that $$\mu_i(m, m+ d(\mu_i) - (m \vee n)) \neq \mu_i(n , n+ d(\mu_i) - (m \vee n)),$$ and thus $\sigma^mx \neq \sigma^nx$ and $x$ is an aperiodic boundary path.  \\

\eqref{F} $\Longrightarrow$ \eqref{W}.  Fix distinct open sets $X, Y \subset \L$ such that $s(X) = s(Y)$ and $s|_X$ and $s|_Y$ are both homeomorphisms.  Similar to \cite[Remark 3.2]{LS}, we may assume $r(X) = r(Y)$ and further that for $\mu \in X$ and $\nu \in Y$ if $s(\mu) = s(\nu)$, then $r(\mu) = r(\nu)$.  To the contrary, suppose there exist $\mu \in X$ and $\nu \in Y$ with $s(\mu) = s(\nu) = v$ but $r(\mu) \neq r(\nu)$.  Then $\MCE(Xv,Yv) = \MCE(\mu, \nu) = \emptyset$.

Also, we may assume that there exist $m \neq n \in \NN^k$ such that $s(X^m) \cap s(Y^n) \neq \emptyset$.  If this were not the case, then for any $\mu \in X$ the unique $\nu \in Ys(\mu)$ must also have degree $m$.  Thus, either $\MCE(Xs(\mu), Ys(\mu)) = \MCE(\mu, \nu) = \emptyset$ or $\mu = \nu$.  Since $X \neq Y$, there must exist a pair with no minimal common extensions.

Now, fix $m \neq n \in \NN^k$ such that $s(X^m) \cap s(Y^n) \neq \emptyset$ as above.  Let $V$ be the nonempty open set $s(X^m) \cap s(Y^n)$.  Choose $\l_{V,m,n}$ by \eqref{F}.  Suppose that $\alpha \in \MCE(X\l, Y\l)$.  Then, $\alpha = \mu\l\mu' = \nu\l\nu'$ where $\mu \in X^m$, $\nu \in Y^n$, and $\mu', \nu' \in s(\l)\L$.  Notice that
\begin{align*}\l(m, m + d(\l) - (m \vee n)) & = \alpha(m + n, m + n + d(\l) - (m \vee n))\\ & = \l(n, n + d(\l) - (m \vee n)),\end{align*}
but this contradicts the choice of $\l$.  Thus, $\MCE(X\l, Y\l) = \emptyset$.

\eqref{W} $\Longrightarrow$ \eqref{F}.  Fix an open set $V \subset \L^0$ and a pair $m \neq n \in \NN^k$.  Notice $r_{m \vee n}^{-1}(V) \subset \L^{m \vee n}$ is open and let $U \subset r_{m \vee n}^{-1}(V)$ such that $s|_U$ is a homeomorphism.  Define $X:= \left\{\mu(m, m \vee n)\,|\,\mu \in U\right\}$ and $Y := \left\{\mu(n, m\vee n)\,|\, \mu \in U\right\}$.  Using a restriction of the continuous and open composition map, it can be shown that $X$ and $Y$ are both open.  As $s(X) = s(U) = s(Y)$ and both $s|_X$ and $s|_Y$ are homeomorphisms, by $\eqref{W}$ choose $\tau \in s(X)\L$ such that $\MCE(X\tau, Y\tau) = \emptyset$.  Since $s|_U$ is a homeomorphism, there is a unique $\mu \in U$ such that $s(\mu) = r(\tau)$.  Let $\xi = \mu(m, m \vee n) \in X$ and $\upsilon = \mu(n, m \vee n) \in Y$ and notice $X\tau = \{\xi\tau\}$ and $Y\tau = \{\upsilon\tau\}$.  Let $\omega \in s(\tau)\L^{m \vee n}$ and define $\l = \mu\tau\omega$.  Now, \begin{align*} \l(m, m + d(\l) - (m \vee n)) & = (\mu\tau\omega)(m, m + d(\tau) + (m \vee n)) \\ & = \xi(\tau\omega)(0, d(\tau) + m) \\ & = \xi\tau\omega(0, m)\end{align*}Similarly, $$\l(n, n + d(\l) - (m \vee n)) = \upsilon\tau\omega(0, n).$$  If $\xi\tau\omega(0,m) = \upsilon\tau\omega(0, n)$, then $\l \in \MCE(\xi\tau, \upsilon\tau) = \MCE(X\tau, Y\tau)$, a contradiction.
\end{proof}

To get a better intuition for Conditions \eqref{F} and \eqref{W}, it is helpful to consider the special case of a directed graph.  If a directed graph $E$ satisfies Condition \eqref{F}, then for any two natural numbers $m \neq n$ we can find a path in $E$ such that the $m^{\text{th}}$ and $n^{\text{th}}$ edges are different.  There is also an enlightening diagram of Condition \eqref{F} in the discrete case in Appendix A of \cite{RandS}.  Consider two paths $\mu$ and $\nu$ in a directed graph $E$ where $s(\mu) = s(\nu)$, $r(\mu) = r(\nu)$ and $d(\mu) \wedge d(\nu) = 0$ as described in \cite[Remark 3.2]{LS}.  This implies that one path, say $\mu$, is just a vertex and $\nu$ must be a loop based at that vertex.  Condition \eqref{W} then says that the loop $\mu$ must have an entry.

\section{Examples}

To demonstrate the importance of these formulations of aperiodicity we give an example of a topological $k$-graph in which Condition \eqref{W} is straightforward to verify whereas Condition \eqref{Y} would be quite difficult.  First we define what we call a \emph{twisted product topological $k$-graph}.  This twisting construction takes a discrete $k$-graph and puts a copy of an appropriate topological space $X$ at each vertex, and twists the edges according to a functor $\tau$.

\begin{prop}\label{twist}

Let $(\L, d)$ be a finitely aligned $k$-graph with no sources, $X$ be a second countable, locally compact, Hausdorff space, and $$\tau:\L \rightarrow \left\{\phi:X \rightarrow X\,|\,\phi \text{ is a local homeomorphism}\right\}$$ a continuous functor.  Then the pair $\big(\L \times_\tau X, \tilde{d}\big)$, with object and morphism sets
$$\Obj(\L \times_\tau X) := \Obj(\L) \times X \text{  and  } \Mor(\L \times_\tau X) := \Mor(\L) \times X,$$
range and source maps
$$r(\l, x) := \left(r(\l), \tau_\l(x)\right) \text{  and  } s(\l, x) := \left(s(\l), x\right),$$
and composition
$$(\l, \tau_\mu(x))\circ(\mu, x) = (\l\mu , x),$$
whenever $s(\l) = r(\mu)$ in $(\L,d)$, and degree functor
$$\tilde{d}(\l,x) = d(\l)$$
is a topological $k$-graph.

\end{prop}

The proof that $\big(\L \times_\tau X, \tilde{d}\big)$ satisfies all the requirements of a topological $k$-graph is straightforward and the details unenlightening, so we omit the proof here.

Consider the discrete 2-graph of Evans and Sims \cite{ES} whose 1-skeleton is pictured in Figure~\ref{1skel}.

\begin{figure}[h]
\begin{center}
\begin{tikzpicture}

\node[inner sep=1pt, circle, fill=black] (v0) at (0,0) {};
	\node[inner  sep=4pt, anchor = north] at (v0.south) {\footnotesize{$v_{\text{\tiny{0}}}$}};
\node[inner sep=1pt, circle, fill=black] (v1) at (2,0) {};
	\node[inner  sep=4pt, anchor = north] at (v1.south) {\footnotesize{$v_{\text{\tiny{1}}}$}};
\node[inner sep=1pt, circle, fill=black] (v2) at (4,0) {};
	\node[inner  sep=4pt, anchor = north] at (v2.south) {\footnotesize{$v_{\text{\tiny{2}}}$}};
\node[inner sep=1pt, circle, fill=black] (v3) at (6,0) {};
	\node[inner  sep=4pt, anchor = north] at (v3.south) {\hspace{.05in}\footnotesize{$v_{\text{\tiny{3}}}$}};
	
\node at (6.5,0) {$\cdots$};
\node at (9.5,0) {$\cdots$};

\node[inner sep=1pt, circle, fill=black] (vn) at (7,0) {};
	\node[inner  sep=4pt, anchor = north] at (vn.south west) {\footnotesize{$v_{\text{\tiny{n}}}$} \hspace{.05in}};
\node[inner sep=1pt, circle, fill=black] (vn+1) at (9,0) {};
	\node[inner  sep=4pt, anchor = north] at (vn+1.south east) {\hspace{.18in}\footnotesize{$v_{\text{\tiny{n+1}}}$}};	
	

\draw[style=semithick, green!50!black, -latex]  (v1.north west) parabola[parabola height = .25 cm] (v0.north east);

\draw[style=semithick, green!50!black, -latex]  (v2.north west) parabola[parabola height = .25 cm] (v1.north east);
\draw[style=semithick, green!50!black, -latex]  (v2.north west) parabola[parabola height = .75 cm] (v1.north east);

\draw[style=semithick, green!50!black, -latex]  (v3.north west) parabola[parabola height = .25 cm] (v2.north east);
\draw[style=semithick, green!50!black, -latex]  (v3.north west) parabola[parabola height = .75 cm] (v2.north east);
\draw[style=semithick, green!50!black, -latex]  (v3.north west) parabola[parabola height = 1.25 cm] (v2.north east);

\draw[style=semithick, green!50!black, -latex]  (vn+1.north west) parabola[parabola height = .25 cm] (vn.north east);
\draw[style=semithick, green!50!black, -latex]  (vn+1.north west) parabola[parabola height = .75 cm] (vn.north east);
\draw[style=semithick, green!50!black, -latex]  (vn+1.north west) parabola[parabola height = 1.75 cm] (vn.north east);

\node at (8, 1.5) {$\vdots$};
	

\node[inner sep=2pt, anchor = south] at (1,.25) {\text{\footnotesize{\color{green!50!black} {$\alpha_{\text{\tiny{0}}}^{\text{\tiny{1}}}$}}}};	

\node[inner sep=2pt, anchor = south] at (3,.25) {\text{\footnotesize{\color{green!50!black} {$\alpha_{\text{\tiny{0}}}^{\text{\tiny{2}}}$}}}};	
\node[inner sep=2pt, anchor = south] at (3,.75) {\text{\footnotesize{\color{green!50!black} {$\alpha_{\text{\tiny{1}}}^{\text{\tiny{2}}}$}}}};	

\node[inner sep=2pt, anchor = south] at (5,.25) {\text{\footnotesize{\color{green!50!black} {$\alpha_{\text{\tiny{0}}}^{\text{\tiny{3}}}$}}}};	
\node[inner sep=2pt, anchor = south] at (5,.75) {\text{\footnotesize{\color{green!50!black} {$\alpha_{\text{\tiny{1}}}^{\text{\tiny{3}}}$}}}};	
\node[inner sep=2pt, anchor = south] at (5,1.25) {\text{\footnotesize{\color{green!50!black} {$\alpha_{\text{\tiny{2}}}^{\text{\tiny{3}}}$}}}};	

\node[inner sep=2pt, anchor = south] at (8,.25) {\text{\footnotesize{\color{green!50!black} {$\alpha_{\text{\tiny{0}}}^{\text{\tiny{n}}}$}}}};	
\node[inner sep=2pt, anchor = south] at (8,.75) {\text{\footnotesize{\color{green!50!black} {$\alpha_{\text{\tiny{1}}}^{\text{\tiny{n}}}$}}}};	
\node[inner sep=2pt, anchor = south] at (8,1.75) {\text{\footnotesize{\color{green!50!black} {$\alpha_{\text{\tiny{n-1}}}^{\text{\tiny{n}}}$}}}};	


\draw[dotted, style=semithick, blue, -latex]  (v1.south west) parabola[parabola height = -.25 cm] (v0.south east);

\draw[dotted, style=semithick, blue, -latex]  (v2.south west) parabola[parabola height = -.25 cm] (v1.south east);
\draw[dotted, style=semithick, blue, -latex]  (v2.south west) parabola[parabola height = -.75 cm] (v1.south east);

\draw[dotted, style=semithick, blue, -latex]  (v3.south west) parabola[parabola height = -.25 cm] (v2.south east);
\draw[dotted, style=semithick, blue, -latex]  (v3.south west) parabola[parabola height = -.75 cm] (v2.south east);
\draw[dotted, style=semithick, blue, -latex]  (v3.south west) parabola[parabola height = -1.25 cm] (v2.south east);

\draw[dotted, style=semithick, blue, -latex]  (vn+1.south west) parabola[parabola height = -.25 cm] (vn.south east);
\draw[dotted, style=semithick, blue, -latex]  (vn+1.south west) parabola[parabola height = -.75 cm] (vn.south east);
\draw[dotted, style=semithick, blue, -latex]  (vn+1.south west) parabola[parabola height = -1.75 cm] (vn.south east);

\node at (8, -1.25) {$\vdots$};


\node[inner sep=2pt, anchor = north] at (1,-.25) {\text{\footnotesize{\color{blue} {$\beta_{\text{\tiny{0}}}^{\text{\tiny{1}}}$}}}};	

\node[inner sep=2pt, anchor = north] at (3,-.25) {\text{\footnotesize{\color{blue} {$\beta_{\text{\tiny{0}}}^{\text{\tiny{2}}}$}}}};	
\node[inner sep=2pt, anchor = north] at (3,-.75) {\text{\footnotesize{\color{blue} {$\beta_{\text{\tiny{1}}}^{\text{\tiny{2}}}$}}}};	

\node[inner sep=2pt, anchor = north] at (5,-.25) {\text{\footnotesize{\color{blue} {$\beta_{\text{\tiny{0}}}^{\text{\tiny{3}}}$}}}};	
\node[inner sep=2pt, anchor = north] at (5,-.75) {\text{\footnotesize{\color{blue} {$\beta_{\text{\tiny{1}}}^{\text{\tiny{3}}}$}}}};	
\node[inner sep=2pt, anchor = north] at (5,-1.25) {\text{\footnotesize{\color{blue} {$\beta_{\text{\tiny{2}}}^{\text{\tiny{3}}}$}}}};

\node[inner sep=2pt, anchor = north] at (8,-.25) {\text{\footnotesize{\color{blue} {$\beta_{\text{\tiny{0}}}^{\text{\tiny{n}}}$}}}};	
\node[inner sep=2pt, anchor = north] at (8,-.75) {\text{\footnotesize{\color{blue} {$\beta_{\text{\tiny{1}}}^{\text{\tiny{n}}}$}}}};	
\node[inner sep=2pt, anchor = north] at (8,-1.75) {\text{\footnotesize{\color{blue} {$\beta_{\text{\tiny{n-1}}}^{\text{\tiny{n}}}$}}}};	

\end{tikzpicture}
\end{center}
\caption{The 1-skeleton of a discrete $k$-graph $\L$.}
\label{1skel}
\end{figure}
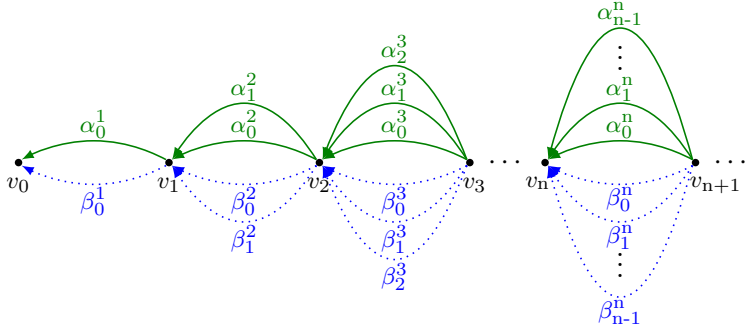
with factorization rules given by
$$\alpha_{i}^{n} \beta_{j}^{n+1} = \beta_{\xi_n(i,j)}^{n} \alpha_{\zeta_n(i,j)}^{n+1}$$
where $(i,j) \mapsto (\xi_n(i,j), \zeta_n(i,j))$ is a permutation of $\ZZ/n\ZZ \times \ZZ/(n+1)\ZZ$.

For an example of a twisted product we take the 1-skeleton of Figure~\ref{1skel}, with $\xi_n$ and $\zeta_n$ being ``plus 1" in the group $\ZZ/n\ZZ$, the topological space $\TT$ with the usual topology, and the functor $\tau$ given by $\tau_{\alpha_i^n}(z) = \tau_{\beta_j^n}(z) := z^n$.

We show the twisted topological $k$-graph above satisfies Condition \eqref{W}.  We start with the appropriate open sets $X, Y \in \L \times_\tau \TT$.  As the topology on $\L$ is discrete, we may assume that $X$ and $Y$ have the form $X = \{\mu\} \times U$ and $Y = \{\nu\} \times V$, for paths $\mu, \nu \in \L$ and open sets $U,V \subset \TT$.  Since $s(X) = s(Y)$, it must be true that $s(\mu) = s(\nu)$ and $U = V$.  As in \cite[Remark 3.2]{LS}, we may assume $r(\mu) = r(\nu)$ and $d(\mu) \wedge d(\nu) = 0$.  We assume $\mu$ consists of only green edges and $\nu$ only blue.  We can see that $\mu$ must have the form $$\mu = \Big(\alpha_{i_0}^m, z^k\Big) \ldots \Big(\alpha_{i_{n-2}}^{m+n-2}, \left(z^{m+n}\right)^{m + n - 1}\Big)\Big(\alpha_{i_{n-1}}^{m+n-1}, z^{m+n}\Big)\Big(\alpha_{i_n}^{m+n}, z\Big)$$
where $z \in U$ and $k$ is given by $\displaystyle{k = \!\!\!\prod_{i = m+1}^{m+n}\!\!i \,}$.  Similarly,$$\nu = \Big(\beta_{j_0}^m, z^k\Big) \ldots \Big(\beta_{j_n}^{m+n}, z\Big)$$
We consider the path $\tau = \left(\beta_{j_0 - n}^{m+n+1}, z^{1/m+n+1}\right)$.  Notice the path $\nu\tau$ is built of all blue edges, so it cannot be factored to appear as a different path.  We follow the factorization rules to rewrite the path $\mu\tau$ such that the first edge is a blue edge, and then compare to $\nu\tau$.
\begin{align*}
\mu\tau & = \Big(\alpha_{i_0}^m, z^k\Big)\Big(\alpha_{i_1}^{m+1}, z^{k/m+n}\Big) \ldots\Big(\alpha_{i_n}^{m+n}, z\Big)\Big(\beta_{j_0 - n}^{m+n+1}, z^{1/m+n+1}\Big)\\
& =\Big(\beta_{j_0 +1}^{m}, z^{k}\Big) \Big(\alpha_{i_0+1}^{m+1}, z^{k/m+n}\Big)\ldots\Big(\alpha_{i_n+1}^{m+n}, z^{1/m+n+1}\Big)
\end{align*}
If $m \neq 1$, $\left(\beta_{j_0}^m, z^k\right) \neq \left(\beta_{j_0+1}^m, z^k\right)$ and the paths $\mu\tau$ and $\nu\tau$ will have no common extensions.  If $m=1$, then we'd amend a similar path of shape $(0,2)$ so that $\mu\tau$ and $\nu\tau$ would be guaranteed to differ in the second blue edge by a similar calculation.  Thus, $\MCE(X\tau,Y\tau) = \emptyset$ and $\L \times_\tau \TT$ is aperiodic.

In contrast, if we were to check for aperiodicity using Condition \eqref{Y} we would need to consider infinite paths of the form $$x = \left(\alpha_0^i, z\right)\left(\beta_0^{i+1}, z^{1/i+1}\right)\left(\alpha_0^{i+2}, z^{1/(i+1)(i+2)} \right)\left(\beta_0^{i+3}, z^{1/(i+1)(i+2)(i+3)}\right) \ldots$$
and investigate $\sigma^mx$ for any $m \in \NN^k$.  A precise calculation requires a variety of formulas for $\sigma^mx$ to cover different cases and then checking the various factorizations of the result for periodic behavior.  It is at least three long and somewhat complicated calculations.

It is worth noting that this work covers only topological $k$-graphs without sources.  This was a common hypothesis in the early works on direct graphs, topological graphs, and $k$-graphs that was eventually replaced with less restrictive hypotheses.  It is the hope of the author to develop similar conditions and results for topological $k$-graphs with sources.

\bibliographystyle{amsplain}
\bibliography{References}

\end{document}